\documentclass[10pt]{amsart}

\usepackage{amsmath}
\usepackage{amsthm}
\usepackage{amsfonts}
\usepackage{amssymb}

\usepackage[colorlinks,citecolor=blue,linkcolor=blue]{hyperref}

\usepackage{amscd}	
\usepackage{cancel}	

\usepackage{graphicx}
\usepackage{color}

\ifdefined\draftversion
    
    \usepackage[shortalphabetic,initials,backrefs]{amsrefs}
\else
    \usepackage[shortalphabetic,initials]{amsrefs}
\fi

\renewcommand{\MultipleCiteKeyWarning}[2]{}

\makeatletter
\def\blfootnote{\xdef\@thefnmark{}\@footnotetext}
\makeatother

\IfFileExists{.git/git.tex}{
\input{.git/git.tex}

}
{

\newcommand{\gitshash}{NA}

\newcommand{\gitauthsdate}{NA}
}

\newcommand{\dimension}{n}

\newcommand{\dist}{\operatorname{dist}}
\newcommand{\diam}{\operatorname{diam}}
\newcommand{\supp}{\operatorname{supp}}
\newcommand{\interior}{\operatorname{int}}

\newcommand{\esssup}{\operatorname*{ess\,sup}}
\newcommand{\essinf}{\operatorname*{ess\,inf}}

\newcommand{\argmax}{\operatorname*{arg\,max}}

\newcommand{\divo}{\operatorname{div}}

\newcommand{\dx}{\;dx}

\newcommand{\diff}[1]{\;d{#1}}

\newcommand{\abs}[1]{\left|#1\right|}
\newcommand{\pth}[1]{\left(#1\right)}
\newcommand{\bra}[1]{\left[#1\right]}
\newcommand{\set}[1]{{\left\{#1\right\}}}

\newcommand{\norm}[1]{\left\|#1\right\|}

\newcommand{\cl}[1]{\overline{#1}}	

\newcommand{\e}{\ensuremath{\varepsilon}}

\newcommand{\R}{\ensuremath{\mathbb{R}}}

\newcommand{\Rd}{\ensuremath{{\mathbb{R}^{\dimension}}}}
\newcommand{\Rn}{\Rd}
\newcommand{\Z}{\ensuremath{\mathbb{Z}}}

\newcommand{\T}{\ensuremath{\mathbb{T}}}
\newcommand{\Tn}{{\T^\dimension}}

\renewcommand{\labelenumi}{(\alph{enumi})}
\renewcommand{\theenumi}{(\alph{enumi})}

\numberwithin{equation}{section}

\newtheorem{theorem}{Theorem}[section]
\newtheorem{lemma}[theorem]{Lemma}
\newtheorem{proposition}[theorem]{Proposition}
\newtheorem{corollary}[theorem]{Corollary}

\newtheorem{definition}[theorem]{Definition}

\theoremstyle{definition}

\newtheorem{remark}[theorem]{Remark}

\usepackage{tikz}

\newcommand{\sign}{\operatorname{sign}}
\newcommand{\E}{\mathcal{E}}
\newcommand{\aff}{\operatorname{aff}}
\newcommand{\TT}{\mathcal{T}}

\newcommand{\SW}{\sigma^{\rm sl}}
\newcommand{\SE}{\E^{\rm sl}}
\newcommand{\hp}{{\hat p}}

\definecolor{darkgreen}{rgb}{0,0.4,0}

\title[Crystalline MCF with a nonuniform driving force term]{Viscosity solutions for the crystalline mean curvature flow with a nonuniform driving force term}

\author[Y. Giga]{Yoshikazu Giga}
\address[Y. Giga]{Graduate School of Mathematical Sciences, University of Tokyo, 3-8-1 Komaba Meguro-ku,
Tokyo 153-8914, Japan.}
\email{labgiga@ms.u-tokyo.ac.jp}
\author[N. Po\v{z}\'{a}r]{Norbert Po\v{z}\'{a}r}
\address[N. Po\v{z}\'{a}r]{Faculty of Mathematics and Physics, Institute of Science and Engineering, Kanazawa University,
Kakuma town, Kanazawa, Ishikawa 920-1192, Japan.}
\email{npozar@se.kanazawa-u.ac.jp}

\date{\today\ (git: \gitauthsdate, \gitshash)}

\subjclass[2010]{}

\begin{document}

\begin{abstract}
A general purely crystalline mean curvature flow equation with a nonuniform driving force term is considered.  The unique existence of a level set flow is established when the driving force term is continuous and spatially Lipschitz  uniformly in time. By introducing a suitable notion of a solution a comparison principle of continuous solutions is established for equations including the level set equations. An existence of a solution is obtained by stability and approximation by smoother problems. A necessary equi-continuity of approximate solutions is established. It should be noted that the value of crystalline curvature may depend not only on the geometry of evolving surfaces but also on the driving force if it is spatially inhomogeneous.
\end{abstract}

\maketitle

\tableofcontents

\section{Introduction}

In our previous works \cite{GP_ADE,GP_CPAM}, we constructed a unique global-in-time level set flow for the crystalline mean curvature flow of the form
\[
	V = g(\nu, \kappa_\sigma).
\]
Here $V$ is the normal velocity of an evolving hypersurface in $\mathbb{R}^n$, $n \geq 2$, in the direction of a unit normal vector field $\nu$ and $\kappa_\sigma$ is a (purely) crystalline mean curvature of the hypersurface. The anisotropy $\sigma$ is assumed to be crystalline, that is, $\sigma: \Rn \to \R$ is a positively one-homogeneous function such that $\set{\sigma < 1}$ is a bounded convex polytope.
 The function $g \in C(\mathcal S^{n-1} \times \R)$ is a given function that is non-decreasing in the second variable so that the problem is at least formally degenerate parabolic; here, $\mathcal S^{n-1}$ denotes the unit sphere in $\Rn$. We are using the convention that $V = \kappa_\sigma$ is the usual mean curvature flow when $\sigma$ is isotropic so that $\kappa_\sigma$ is the usual mean curvature.

Our goal is to extend the result in \cite{GP_ADE} to the problem
\begin{align}
\label{mcf}
V = g\big(\nu, \kappa_\sigma + f(x,t)\big),
\end{align}
where $f=f(x,t)$ is a continuous function that is Lipschitz continuous in space variable $x$ uniformly in time $t$.
 Namely, we consider a crystalline mean curvature flow with a nonuniform driving force term.
 We introduce a suitable notion of viscosity solutions to the level set equation for~\eqref{mcf}, which looks slightly weaker than those in \cite{GP_ADE,GP_CPAM}.
 Our main result reads:
\begin{theorem}
\label{th:existence-flow}
Assume that $g \in C(\mathcal S^{n-1} \times \R)$ is Lipschitz continuous in the second variable uniformly in the first variable and non-decreasing in the second variable, $\sigma$ is a crystalline anisotropy and $f \in C(\Rn \times \R)$ is Lipschitz continuous in space uniformly in time.
 Then there is a unique global-in-time level set flow to \eqref{mcf} when the initial hypersurface is compact.
\end{theorem}

The assumption on $g$ prohibits superlinear growth in $\kappa_\sigma + f$.
 However, it is still quite general since it allows nonlinear dependence in $\kappa_\sigma + f$.

The general strategy to prove this result is along the line of \cite{GP_CPAM}.
 However, the problem is substantially more difficult when $f$ is spatially nonuniform even if $f$ is time-independent.

We have to understand $\kappa_\sigma + f$ at a given time $t$ as one term given as the canonical restriction of the subdifferential of the functional
\begin{align*}
\mathcal{G}_t(E) := \int_{\partial E} \sigma(\nu) \diff S + \int_E f_t\dx,
\end{align*}
where $f_t := f(\cdot, t)$, to be consistent with the formal gradient flow structure, with Chambolle et al. \cite{CMP,CMNP} and previous work in 1D in \cite{GGRybka_ARMA,GGNakayasu_GeomPDE}; see also \cite[Section~2.3]{BraidesMalusaNovaga}. This becomes important in the proof of stability in Section~\ref{sec:stability}.

To establish Theorem~\ref{th:existence-flow}, we study the well-posedness of the level set formulation for \eqref{mcf}.
Following the convention of \cite{GP_CPAM}, we take a level set function $u = u(x,t)$ such that
its every sublevel set is the solution of \eqref{mcf}.
Then $u$ is a solution of
\begin{align}
\label{pde}
u_t + F\big(\nabla u, \divo(\nabla\sigma(\nabla u)) - f\big) = 0,
\end{align}
where
\begin{align}
\label{level-set-F}
F(p,\xi) := |p| g\pth{\frac p{|p|}, - \xi}.
\end{align}

Let us assume that $f = f(x)$ for now to simplify the notation.  To define $\kappa_\sigma + f$ for admissible faceted functions on their facets, we take
\begin{align}
\label{Ef}
\E_f(\psi) := \int \sigma(\nabla \psi) + f\psi\dx \qquad \text{for } \psi \in Lip,
\end{align}
where the integral is taken over some appropriate domain like $\T^n$,
and define
\begin{align*}
\Lambda_f[\psi] = \divo z_{\rm min} - f = - \partial^0 \E_f(\psi),
\end{align*}
where $z_{\rm min}$ is a minimizer of
\begin{align*}
\norm{\divo z - f}_{L^2}
\end{align*}
over all $z \in L^\infty$, $\divo z \in L^2$ such that
\begin{align*}
z \in \partial \sigma(\nabla \psi) \qquad \text{a.e.}
\end{align*}

Note that $\partial \E_f(\psi) = \partial \E_0(\psi) + f$. Therefore $f$ does not change admissible functions or admissible facets, only the value of the canonical restriction $\partial^0 \E_f(\psi)$. We will see below that $\Lambda_f$ satisfies a comparison principle, Proposition~\ref{pr:Lambda-comparison}.

We define a notion of viscosity solutions for \eqref{pde} (Definition~\ref{def:visc-solution}), which is a generalization of the notion introduced in the previous work \cite{GP_ADE,GP_CPAM} to allow for the nonuniform driving force term.
We first establish a comparison principle (Theorem~\ref{th:comparison-principle}) for equations including the level set equation of \eqref{mcf}. However, our flattening argument \cite{GG_ARMA_graphs,GP_ADE,GP_CPAM} requires that one of sub- and supersolutions is continuous.
 This requirement is unnecessary when $n=2$ since the set of singular directions of the interfacial energy $\sigma$ is compact \cite{GGRybka_ARMA}.
 We next prove the existence of a solution by showing the stability for relaxed limits of solutions of \eqref{pde} with regularized $\sigma$: both with quadratic growth (Theorem~\ref{th:stability}) and with linear growth (Theorem~\ref{th:stability-linear-growth}).
 However, to show the full convergence through the comparison principle we need to show that the relaxed semilimit must be continuous.
 For this purpose, we establish a uniform Lipschitz bound in space (Theorem~\ref{L}) and a uniform 1/2-H\"{o}lder bound in time (Theorem~\ref{H}) for an approximate solution.
 Although a Lipschitz bound is well known in the elliptic case even for viscosity solutions \cite{Barles}, it is not trivial to adapt it in the parabolic case, especially in our setting. We shall give a direct proof for a spatial Lipschitz bound for viscosity solutions for level set equations without appealing to the classical theory of quasilinear or fully nonlinear uniformly parabolic equations; see e.~g.~\cite{Li,LSU,Lu}. See, for example, also \cite{AG,GOS} for applications to viscosity solutions.  We also give $1/2$-Hölder bound in time by constructing suitable barriers.
Using these results, we deduce the unique existence of solutions of \eqref{pde} (Theorem~\ref{th:existence}). The proof of Theorem~\ref{th:existence-flow} is outlined at the end of Section~\ref{sec:lipschitz-bound}.

\subsection*{Literature overview.}
 The crystalline mean curvature flow was introduced in mathematical community by Angenent and Gurtin \cite{AGu} and independently by Taylor \cite{T91} around thirty years ago.
 Since then, there is a large number of literature.
 The bibliography of \cite{GP_ADE,GP_CPAM} includes several key references on the crystalline mean curvature flow or flow with constant driving force term.
 We here mention references related to the crystalline mean curvature flow with nonuniform driving force term.
 The problem is far more difficult than the case of constant driving force even for the problem of planar motion because the expected speed on facets may not be constant, which may cause facet bending or splitting.
 A first global unique existence result has been established in \cite{GG2} for a graph-like curve based on the theory of maximal monotone operators.
 Several explicit facet splitting solutions are constructed in \cite{GG2}.
 In \cite{BGN} planar, crystalline flow with nonuniform driving force was constructed but under the assumtion that the driving force preserves facets, in other words, facet splitting and bending does not occur.
 Several conditions for the preservation of facets are given especially for the Stefan problem when the anisotropy is fixed like its Wulff shape is a cylinder \cite{GR1}, \cite{GR2}, \cite{GR3}.
 An explicit facet bending solution is given in a planar motion under rectangular anisotropy \cite{GR4}, \cite{GR5}, \cite{GGoR1}, \cite{GGoR2}, \cite{GGoR3}; see also \cite{MR1}, \cite{MR2}.
 For a graph-like curve a general global well-posedness results are established in \cite{GGNakayasu_GeomPDE,GGRybka_ARMA} for a general equation like \eqref{mcf}.
 For a closed curve less is known.
 In \cite{CN} a local-in-time strong solution is constructed when the force term is Lipschitz in space-time or spatially independent for the crystalline curvature flow in the plane.
 It is quite recent that the level set approach is extended to higher dimensional setting even for constant driving force \cite{GP_ADE,GP_CPAM}.
 In \cite{CMNP} quite general driving force is allowed but the equation is of the form $V=M(\nu)(\kappa_\sigma+f)$ with ``convex'' mobility.
 They proved the unique existence of the level set flow for any initial hypersurface which may be unbounded.
 In our notion, they assume that $g$ is linear in the last variable.
 Although they allow general anisotropy $\sigma$, not necessarily purely crystalline, their method essentially depends on linearity of $g$ in the last variable.
 Their method is substantially different from ours.

\medskip
This paper is organized as follows.
 In Section~\ref{se:visc-sol}, we introduce a notion of a viscosity solution.
 In Section~\ref{se:comp-principle}, a comparison principle is established.
 In Section~\ref{sec:stability}, we prove stability of a solution.
 In Section~\ref{sec:lipschitz-bound}, we prove a necessary Lipschitz bound in space and $1/2$-H\"{o}lder bound in time to show the continuity of the limit.
 In Section~\ref{sec:proof-of-nonexistence}, we warn that our value $\kappa_\sigma + f$ is not a simple sum of $\kappa_\sigma$ and $f$.
 In the Appendix~\ref{sec:resolvent-problem}, we give a proof of a Lipschitz bound of our resolvent problem.

\section{Viscosity solutions}
\label{se:visc-sol}

\subsection{Facet}

Let $\sigma$ be an anisotropy, that is, let $\sigma: \Rn \to \R$ be positively one-homogeneous function such that $\set{\sigma < 1}$ is a bounded convex set.
Suppose that $U \subset \Rn$ is an open set and $\psi \in Lip(U)$. We define the set of Cahn-Hoffman vector fields for $\psi$ as
\begin{align*}
CH(\psi; U) := \set{z \in X^2(U): z \in \partial \sigma(\nabla \psi) \text{ a.~e.}}.
\end{align*}
Here $X^2(U) := \set{z \in L^\infty(U): \divo z \in L^2(U)}$; see \cite{Anzellotti}.
If $CH(\psi; U)$ is nonempty, we define the $\sigma^\circ$-($L^2$) divergence of $\psi$ for any $f \in L^2(U)$ as
\begin{align}
\label{Lipschitz-min-divergence}
\Lambda_f[\psi] := \divo z_{\rm min} - f \qquad \text{on } \set{\psi = 0},
\end{align}
where $z_{\rm min} \in CH(\psi; U)$ minimizes $z \mapsto \norm{\divo z - f}_{L^2(U)}$. Note that since \[\divo CH(\psi; U) := \set{\divo z: z \in CH(\psi; U)}\]  is a closed, convex and nonempty (by assumption) subset of $L^2(U)$ there exists a unique minimizing $\divo z_{\rm min}$ (but $z_{\rm min}$ might not be unique).

Let us recall the comparison principle for $\sigma^\circ$-($L^2$) minimal divergence, proved in \cite{GP_ADE} for $f \equiv 0$. The generalization to $f \not\equiv 0$ is straightforward.

\begin{proposition}[{cf. \cite[Proposition~4.12]{GP_ADE}}]
  \label{pr:Lambda-comparison}
  If $\psi_1, \psi_2$ are two
  Lipschitz functions on an open set $U$ such that their zero sets are compact subsets of $U \subset \Rn$, $f_1, f_2 \in L^2(U)$ and
  $\Lambda_{f_i}[\psi_i]$, $i = 1,2$, are well-defined, then
  \begin{align*}
    \sign \psi_1 \leq \sign \psi_2 \quad \text{on $U$} \quad \text{and} \quad f_1 \geq f_2 \quad \text{a. e. on $U$}
  \end{align*}
  imply
  \begin{align*}
    \Lambda_{f_1}[\psi_1] \leq \Lambda_{f_2}[\psi_2], \qquad \text{a. e. on } \set{\psi_1 =
    0} \cap \set{\psi_2 = 0}.
  \end{align*}
\end{proposition}

The dependence of the minimal divergence only on $\sign \psi$ motivates the following definition.
Let $\mathcal F = \set{\xi\mid \xi: \Rn \to \R}$ be the set of all real-valued functions on $\Rn$. We define the relation on $\mathcal F$ as
\begin{align*}
\xi_1 \sim \xi_2\qquad  \Leftrightarrow \qquad \sign \xi_1 = \sign \xi_2,
\end{align*}
where $\sign s = +1, 0, -1$, respectively, when $s > 0, s = 0, s < 0$.
The relation $\sim$ defines an equivalence relation on $\mathcal F$. We refer to the equivalence classes $[\xi] := \set{\psi: \psi \sim \xi}$ of $\mathcal F$ with respect to $\sim$ as \emph{(abstract) facets}. We write $[\xi_1] \preceq [\xi_2]$ when $\sign \xi_1 \leq \sign \xi_2$ and this relation defines a total order on the set of all facets $\mathcal F / \sim := \set{[\xi]: \xi \in \mathcal F}$.

We say that a facet $[\xi]$ is a \emph{$\sigma^\circ$-($L^2$) Cahn-Hoffman facet} if $\set{\xi = 0}$ is compact and there are an open set $U \subset \Rn$, $\set{\xi = 0} \subset U$, and a Lipschitz function $\psi \in [\xi]$ such that $CH(\psi; U) \neq \emptyset$.

\begin{remark}
\label{rem:periodic-extension}
The definition of $\sigma^\circ$-($L^2$) Cahn-Hoffman facet guarantees that there exists a periodic Lipschitz function $\zeta$ on $\Rn$ and a periodic vector field $z \in L^\infty(\Rn)$ such that $\zeta = \psi$ in a neighborhood of $\set{\xi = 0}$, $z \in \partial \sigma(\nabla \zeta)$ a.~e. and $\divo z \in L^2_{\rm loc}(\Rn)$. Indeed, we can assume that $U$ is bounded and hence $\partial U$ is compact. We set $\e := \min_{\partial U} |\psi| / 2 > 0$ and consider $\theta := \min(\e, \max(-\e, \psi))$. Clearly $z \in \partial \sigma(\nabla \theta)$ a.~e. as $\partial \sigma(p) \subset \partial \sigma(0)$ for any $p \in \Rn$. When $n \geq 2$, either $\theta \equiv \e$ or $\theta \equiv -\e$ outside of a large ball and hence a periodic extension is trivial. When $n = 1$ we might have to modify $\theta$ by an even extension with respect to some point outside of $U$ to guarantee that $\theta$ has one sign outside of a large ball. For details see \cite{GP_ADE}. This allows us to consider facets as objects on $\Tn$.
\end{remark}

For a $\sigma^\circ$-($L^2$) Cahn-Hoffman facet $[\xi]$ and $f \in L^2(\R^n)$ we define the $\sigma^\circ$-($L^2$) minimal divergence of the facet $[\xi]$, $\Lambda[\xi] \in L^2(\set{\xi = 0})$, as
\begin{align}
\label{minimal-divergence}
\Lambda_f[\xi] := \divo z_{\rm min} -f \qquad \text{on } \set{\xi = 0},
\end{align}
where $z_{\rm min} \in CH(\psi; U)$ minimizes $z \mapsto \norm{\divo z - f}_{L^2(U)}$ for some $U \subset \Rn$ open, $\set{\xi = 0} \subset U$, and Lipschitz $\psi \in [\xi]$ such that $CH(\psi; U) \neq \emptyset$. Note that $\Lambda_f[\xi]$ is well-defined since it does not depend on the choice of $\psi$ or $U$ by Proposition~\ref{pr:Lambda-comparison}, and the notation is consistent with \eqref{Lipschitz-min-divergence} since $\psi \in [\psi]$.

If $f$ is locally bounded, by comparison with Wulff functions like $\max(\sigma^\circ - c, 0)$ we can show that $\Lambda_f[\xi]$ is locally bounded on the interior of $\set{\xi = 0}$ by Proposition~\ref{pr:Lambda-comparison} and we can define
\begin{align}
\label{max-Lambda}
\underline\Lambda_f[\xi](x) := \lim_{\delta \searrow 0} \essinf_{B_\delta(x)} \Lambda_f[\xi], \qquad
\overline\Lambda_f[\xi](x) := \lim_{\delta \searrow 0} \esssup_{B_\delta(x)} \Lambda_f[\xi].
\end{align}
These values might differ since $\Lambda_f[\xi]$ is in general discontinuous.

\begin{lemma}
\label{le:lambda-const-dep}
For any $\sigma^\circ$-($L^2$) Cahn-Hoffman facet $[\chi]$ and $f \in L^2$,
\begin{align}
\label{lambda-const-dep}
\Lambda_{f + c}[\chi] = \Lambda_f[\chi] - c \qquad \text{a.e. on $\set{\chi = 0}$ for any constant $c \in \R$.}
\end{align}
\end{lemma}

\begin{proof}
We have $\Lambda_f[\chi] = \divo z_{\rm min} - f$ on $\set{\chi = 0}$ where $z_{\rm min}$ minimizes $z \mapsto \norm{\divo z- f}_{L^2(U)}$ over $z \in CH(\psi; U)$ for some open $U \supset \set{\chi = 0}$ and Lipschitz $\psi \in [\chi]$. By Remark~\ref{rem:periodic-extension} we can assume that $\psi \in Lip(\Tn)$, $\chi \in L^2(\Tn)$. To simplify the notation, let us set $C := \set{\divo z: z \in CH(\psi; \Tn)}$ and let $P_E x$ denote the element of $E$ closest to $x$ for $E \subset L^2(\Tn)$ closed, convex and $x \in L^2(\Tn)$. We have $\Lambda_f[\chi] = P_{C - f} 0 = P_C f - f$ on $\set{\chi = 0}$.

To deduce \eqref{lambda-const-dep}, we only need to show that $P_C f = P_C(f + c)$ for all $c \in \R$. Observe that $C \subset 1^\perp$ where $1^\perp \subset L^2(\Tn)$ is the orthogonal subspace to $1 \in L^2(\Tn)$.
Indeed,  by Green's theorem \cite{Anzellotti}
\begin{align*}
(1, \xi)_{L^2(\Tn)} = (1, \divo z)_{L^2(\Tn)} = 0 \qquad \text{for all }\xi = \divo z \in C.
\end{align*}
In particular, the Pythagorean theorem implies
\begin{align*}
P_C(f + c) = P_C P_{1^\perp} (f + c) = P_C P_{1^\perp} f = P_C f.
\end{align*}
\eqref{lambda-const-dep} follows.
\end{proof}

Let us emphasize that Lemma~\ref{le:lambda-const-dep} holds only for constant $c$. It is not in general true that $\Lambda_f = \Lambda_0 - f$.

\begin{corollary}
\label{co:constant-diff-lambda}
If $\abs{f_1 - f_2} \leq M$ then
\begin{align*}
\abs{\Lambda_{f_1}[\chi] - \Lambda_{f_2}[\chi]} \leq M \qquad \text{a.e. on }\set{\chi = 0}.
\end{align*}
\end{corollary}

\begin{proof}
We have $f_1 \leq f_2 + M$ and so by the comparison principle Proposition~\ref{pr:Lambda-comparison} and the dependence on contants Lemma~\ref{le:lambda-const-dep}, we have
\begin{align*}
\Lambda_{f_1}[\chi] \geq \Lambda_{f_2+M}[\chi] = \Lambda_{f_2}[\chi] - M.
\end{align*}
The other inequality is analogous.
\end{proof}

Although in this paper we do not  use the next Corollary on the stability of $\Lambda$ with respect to $f$, we give it as a simple application of Corollary~\ref{co:constant-diff-lambda}.

\begin{corollary}
\label{co:uniform-convergence}
If $f_k \rightrightarrows f$ then
\begin{align*}
\norm{\Lambda_{f_k}[\chi] - \Lambda_f[\chi]}_{L^\infty(\set{\chi = 0})} \to 0 \qquad \text{as } k \to \infty.
\end{align*}
\end{corollary}

\begin{proof}
Clear from Corollary~\ref{co:constant-diff-lambda}.
\end{proof}

\subsection{Slicing}
\label{sec:slicing}

From this point on we shall assume that $\sigma$ is a \emph{crystalline} anisotropy, that is, the set $\set{\sigma < 1}$ is a polytope. Based on a fixed gradient $\hat p \in \Rn$, we consider the orthogonal decomposition of the space $\Rn = Z \oplus Z^\perp$ so that $Z$ is the linear subspace parallel to the affine hull $\aff \partial \sigma(\hat p)$. In other words, $Z$ is the smallest subspace such that $\partial \sigma(\hat p) \subset Z + \xi$ for some $\xi \in \Rn$. Let $k := \dim \partial \sigma(\hat p) := \dim Z$. By fixing orthogonal bases of $Z$ and $Z^\perp$, we can fix two linear isometries $\TT: \R^k \to Z$ and $\TT_\perp: \R^{n-k} \to Z^\perp$. Then we can write every $x \in \Rn$ uniquely as $x = \TT x' + \TT_\perp x''$, where $x' \in \R^k$ and $x'' \in \R^{n-k}$. For $k = 0$ or $k = n$, we take $x = x''$ or $x = x'$, respectively. Note that $x' = \TT^*x$ and $x'' = \TT_\perp^*x$, where $\TT^*$ and $\TT_\perp^*$ are the adjoints of $\TT$ and $\TT_\perp$, respectively. The precise choice of $\TT$ and $\TT_\perp$ for each $\hat p \in \Rn$ is not important as long as we use it consistently.

Using this decomposition, we introduce the positively one-homogeneous function $\SW_{\hat p}:\R^k \to \R$ as
\begin{align}
  \label{sliced-W}
  \SW_{\hat p}(w) := \lim_{\lambda \to 0+} \frac{\sigma(\hat p + \lambda \TT w) - \sigma(\hat
  p)}{\lambda}, \qquad w \in \R^k.
\end{align}
This function represents the infinitesimal structure of $\sigma$ near $\hat p$, sliced in
the direction of $Z$.

For a $(\SW_{\hat p})^\circ$-$(L^2)$ Cahn-Hoffman facet $[\chi]$ on $\R^k$ and $f \in L^2(\R^k)$, we denote $\Lambda_f[\chi]$ as defined in \eqref{minimal-divergence} with $\sigma = \SW_{\hat p}$ on $\R^k$ by $\Lambda_{\hat p, f}[\chi]$.

With $\hat p$, $k$ as above, we say that a function $\psi \in Lip(\R^k)$ is a \emph{$\hat
p$-admissible support function} if $[\psi]$ is a $(\SW_{\hat
p})^\circ$-$(L^2)$ Cahn-Hoffman facet.

\subsection{Definition of viscosity solutions}

We generalize the definition from \cite{GP_CPAM}. We recall the definition of test functions. The variables $x' = \TT^* x$, $x'' = \TT_\perp^* x$ refer to the sliced decomposition introduced in Section~\ref{sec:slicing}.

\begin{definition}[{cf. \cite[Definition~4.7]{GP_CPAM}}]
  \label{de:admissible-stratified-test-func}
Let $\hat p \in \Rn$, $(\hat x, \hat t) \in \Rn \times \R$, $k := \dim \partial \sigma(\hat p)$.
We say that $\varphi(x,t) = \psi(x') + \theta(x'') + \hat p \cdot x + g(t)$ is an \emph{admissible
stratified faceted test function at $(\hat x, \hat t)$ with slope $\hat p$} if $\theta \in C^1(\R^{n -
k})$, $\nabla \theta(\hat x'') = 0$, $g \in C^1(\R)$, and $\psi \in Lip(\R^k)$ is a $\hat p$-admissible
support function with $\hat x' \in \interior \set{\psi = 0}$.
Note that if $k = 0$, we have $\varphi(x,t) = \theta(x) + g(t)$ for some $\theta \in C^1(\Rn)$, $g \in
C^1(\R)$.
\end{definition}

\begin{definition}[{Viscosity solution, cf. \cite[Definition~5.2]{GP_ADE}}]
  \label{def:visc-solution}
  We say that an upper semicontinuous function $u$ is a \emph{viscosity subsolution} of
  \begin{align}
  \label{visc-sol-eq}
  u_t + F\Big(x, t, \nabla u, \divo (\nabla \sigma(\nabla u)) - f\Big) = 0
  \end{align}
  on $\Rn \times (0,
  T)$, $T > 0$, if for any $\hat p
  \in \Rn$, $\hat x \in \Rn$, $\hat t \in (0, T)$ and any admissible
  stratified faceted test function $\varphi$ at $(\hat x, \hat t)$ with slope $\hat p$ of the form $\varphi(x, t) = \psi(x') + \theta(x'') + \hat p \cdot x + g(t)$
  such that the function $u - \varphi(\cdot - h)$ has a global maximum on $\Rn \times (0, T)$ at $(\hat x, \hat t)$ for all
  sufficiently small $h' \in \R^k$ and $h'' = 0$, then
  \begin{align}
    \label{visc_subsolution}
    g'(\hat t) + F\Big(\hat x, \hat t, \hat p, \underline{\Lambda}_{\hat p, \hat f}[\psi](\hat x')\Big)  \leq 0,
  \end{align}
  where $\hat f(x') := f(\TT x' + \TT_\perp\hat x'', \hat t)$ and $\underline\Lambda$ is defined in \eqref{max-Lambda}.

  \emph{Viscosity supersolutions} are defined analogously as lower semicontinuous functions, replacing a global maximum with a global
  minimum, $\underline{\Lambda}$ with $\overline{\Lambda}$, and
  reversing the inequality in \eqref{visc_subsolution}.

  A continuous function that is both a viscosity subsolution and a viscosity supersolution is
  called a \emph{viscosity solution}.
\end{definition}

Note that \eqref{visc_subsolution} is weaker than the condition in \cite{GP_ADE,GP_CPAM} since
\[\essinf_{B_\delta(x)} \Lambda_{\hat p, \hat f}[\psi] \leq \underline\Lambda_{\hat p, \hat f}[\psi](x) \qquad \text{for any $\delta > 0$ small}.\]
This is to allow for the dependence of $f$ on $x$ and it will become important in the proof of stability in Section~\ref{sec:stability} since the right-hand side of \eqref{last-visc-cond} is not quite zero. But this relaxed condition \eqref{visc_subsolution} is still strong enough for the comparison principle to hold.

\section{Comparison principle}
\label{se:comp-principle}

We will establish the comparison principle between a viscosity subsolution $u$ and a viscosity supersolution $v$ under the additional assumption that at least one of them is \emph{continuous}. This is enough to show the existence of solutions by approximation as we can obtain a uniform modulus of continuity for the approximating solutions; see Section~\ref{sec:lipschitz-bound}. At this time we do not know how to establish the comparison principle for semicontinuous solutions.

In this section we allow for an explicit dependence of $F$ on the variables $x$ and $t$.
If we do not assume further regularity on at least one of the solutions, we will need to also assume that $F$ satisfies
\begin{align}
\label{F-regularity}
|F(x, t, p, \xi) - F(y, s, p, \eta)| \leq C(|p| + 1)(|x - y| + |t -s| + |\xi - \eta|)
\end{align}
for some constant $C$.
This covers both crystalline mean curvature flows with
\begin{align*}
F(p, \xi) = -\beta(p) \xi
\end{align*}
and anisotropic total variation flows
\begin{align*}
F(p, \xi) = - \xi,
\end{align*}
but forbids superlinear growth of $F$ in the last variable.

\begin{theorem}[Comparison principle]
\label{th:comparison-principle}
Let $\sigma$ be a crystalline anisotropy, $F \in C(\Rn \times \R \times \Rn \times \R)$ be non-increasing in the last variable, and $f \in C(\Rn \times \R)$ be Lipschitz continuous in space, uniformly in time.
Let $u$ be a viscosity subsolution of \eqref{pde} and let $v$ be a \emph{continuous} viscosity supersolution of \eqref{pde} on $Q_T := \Rn \times (0, T)$ for some $T > 0$ in the sense of Definition~\ref{def:visc-solution}. Suppose that $u$ and $v$ are bounded and that there exist constants $R> 0$,  $a \leq b$ such that $u \equiv a$ and $v \equiv b$ on $(\Rn \setminus B_R(0)) \times (0, T)$. Suppose that either
\begin{enumerate}
\renewcommand{\labelenumi}{(\roman{enumi})}
\renewcommand{\theenumi}{\roman{enumi}}
\item $v$ is Lipschitz continuous in space, uniformly in time; or
\label{v-Lipschitz}
\item $F$ satisfies \eqref{F-regularity} and $f$ is Lipschitz continuous in both variables.
\label{f-Lipschitz}
\end{enumerate}

Then
\begin{align}
\label{comp-init}
u(\cdot, 0) \leq v(\cdot, 0) \qquad \text{on } \Rn
\end{align}
implies
\begin{align*}
u \leq v\qquad \text{on } Q_T.
\end{align*}
\end{theorem}

\begin{proof}
Suppose that $\sup_{Q_T} u - v > 0$. Let us fix $\mu > 0$ such that
\begin{align*}
\sup_{Q_T} w  =: m_0 > 0.
\end{align*}
where
\begin{align*}
w(x,t) = u(x,t) - v(x,t) - \frac{\mu}{T- t}.
\end{align*}
Let $(\hat x, \hat t) \in Q_T$ be such that $w(\hat x, \hat t) = m_0$.

Consider
\begin{align*}
\Psi_{\zeta, \e} (x, t, y, s) := u(x,t) - v(y,s) - \frac{|x - y - \zeta|^2}{2\e} - \frac{|t - s|^2}{2\e} - \frac\mu{T - t}.
\end{align*}
For $\e > 0$ small enough and $|\zeta| \leq \sqrt{m_0 \e}$ this function has a maximum on $Q_T \times Q_T$.

\begin{proposition}
\label{pr:point-max-dist}
Assume that $u$ and $v$ satisfy all the assumptions in Theorem~\ref{th:comparison-principle}.
Let $(x_\e, t_\e, y_\e, s_\e)$ be a sequence of maxima of $\Psi_{\zeta_\e, \e}$ for some sequence $|\zeta_\e| \leq \sqrt{m_0 \e}$.
Then there is a constant $M$, independent of $\e$, so that we have
\begin{align}
\label{point-max-dist}
|x_\e- y_\e- \zeta_\e| \leq \sqrt{M \e}, \qquad |t_\e - s_\e|  \leq \sqrt{M\e},
\end{align}
and
\begin{align}
\label{p-limit}
\frac{|x_\e - y_\e - \zeta_\e|^2}{\e} \to 0 \qquad \text{as } \e \to 0.
\end{align}
\end{proposition}

Note that if $\zeta_\e = 0$ the limit \eqref{p-limit} is standard in the viscosity theory; see for example \cite[Proposition~3.7]{CIL}. However, the standard proof does not seem to apply with $\zeta_\e \neq 0$ and we need a continuity of $u$ or $v$ to recover \eqref{p-limit}. The proof of Proposition~\ref{pr:point-max-dist} uses the idea in \cite{DeZanSoravia}. We need \eqref{p-limit} to show that the left-hand side of \eqref{final-comp-ineq} below converges to zero as $\e \to 0$.

\begin{proof}
By maximality
\begin{align*}
\Psi_{\zeta_\e,\e}(x_\e, t_\e, y_\e, s_\e) \geq \Psi_{\zeta_\e, \e} (\hat x, \hat t, \hat x, \hat t) = w(\hat x, \hat t) - \frac{|\zeta_\e|^2}{2\e} \geq \frac{m_0}2 > 0.
\end{align*}
We deduce \eqref{point-max-dist} as $u(x, t) - v(y, s)$ is bounded above by some constant $M$ on $Q_T\times Q_T$.

Since at least one of $(x_\e, t_\e)$ and $(y_\e, s_\e)$ must lie in $B_R(0) \times [0, T)$, we can assume that $(x_\e, t_\e, y_\e, s_\e) \to (\bar x, \bar t) \in \cl{B_R(0)} \times (0, T)$ along a subsequence. Indeed, $\bar t > 0$ by \eqref{comp-init} and $\bar t < T$ since $\sup \Psi_{\zeta_\e, \e}(\cdot, t, \cdot, \cdot) \to -\infty$ as $t \to T-$.

Therefore we have
\begin{align*}
u(\bar x, \bar t) &- v(\bar x - \zeta_\e, \bar t) - \mu(T - \bar t)^{-1} =
\Psi_{\zeta_\e, \e} (\bar x, \bar t, \bar x - \zeta_\e, \bar t)\\
&\leq \Psi_{\zeta_\e, \e} (x_\e, t_\e, y_\e, s_\e)\\
&= u(x_\e, t_\e) - v(y_\e, s_\e) - \frac{|x_\e - y_\e - \zeta_\e|^2}{2\e} - \frac{|t_\e - s_\e|^2}{2\e} - \mu(T - t_\e)^{-1}.
\end{align*}
After rearranging,
\begin{align*}
\frac{|x_\e - y_\e - \zeta_\e|^2}{2\e} + \frac{|t_\e - s_\e|^2}{2\e}  &\leq u(x_\e, t_\e) - u(\bar x, \bar t) + v(\bar x - \zeta_\e, \bar t) - v(y_\e, s_\e)\\
&\qquad + \mu(T- \bar t)^{-1} - \mu(T - t_\e)^{-1}.
\end{align*}
Taking $\limsup_{\e \to 0+}$ of both sides, and using the upper semicontinuity of $u$ and the continuity of $v$, we recover \eqref{p-limit}.
Since every subsequence has a subsequence with limit $0$, we conclude that \eqref{p-limit} holds for the full sequence.
\end{proof}

We showed in \cite[Prop.~7.4 ]{GP_ADE} the following.

\begin{lemma}
There exists $\e_0>0$ such that
for every fixed $0 < \e < \e_0$ there is a set $\Xi_\e \subset \Rn$ on which $\partial \sigma$ is constant, and $\zeta_\e \in \Rn$, $\lambda_\e > 0$ with $|\zeta_\e| + 2 \lambda_\e <\kappa(\e) := \sqrt{m_0 \e}$ such that for every $\zeta$, $|\zeta - \zeta_\e| \leq 2 \lambda_\e$ there is a maximizer $(x, t, y, s)$ of $\Psi_{\zeta_\e, \e}$ on $Q_T \times Q_T$ such that
\begin{align*}
\frac{x - y - \zeta}\e \in \Xi_\e.
\end{align*}
\end{lemma}

Due to a bit of convex analysis explained in \cite{GP_ADE}, $\Xi_\e - \Xi_\e \subset (\aff \partial \sigma(p))^\perp$ for every $p \in \Xi_\e$ and therefore $\aff \partial \sigma(p) \subset (\Xi_\e - \Xi_\e)^\perp$.

We set $Z_\e := \aff \partial \sigma(p)$ for some $p \in \Xi_\e$, $k_\e := \dim Z_\e$, with the linear isometries $\TT_\e: \R^k \to Z_\e$ and $\TT_{\perp, \e}: \R^{n-k} \to Z_\e^\perp$. Since $p \perp Z_\e$ because $\sigma$ is one-homogeneous, the consequence of the flatness lemma, \cite[Lemma~7.6]{GP_ADE}, reduces to:
\begin{lemma}
For $\zeta_\e$, $\lambda_\e$ and $Z_\e$ as above, we have
\begin{align*}
\ell_\e(\zeta) = const \qquad \text{for } \zeta \in \zeta_\e + Z_\e,\ |\zeta - \zeta_\e| \leq 2 \lambda_\e,
\end{align*}
where
\begin{align*}
\ell_\e(\zeta) := \sup_{Q_T \times Q_T} \Psi_{\zeta, \e}.
\end{align*}
\end{lemma}

Choosing some point of maximum $(x_\e, t_\e, y_\e, s_\e) \in Q_T \times Q_T$ of $\Psi_{\zeta_\e, \e}$ and setting $p_\e := (x_\e - y_\e - \zeta_\e)/\e$, we can follow the construction of $p_\e$-admissible faceted test functions in \cite{GP_ADE,GP_CPAM}. This gives us two test functions $\varphi_u$, $\varphi_v$ for $u$ and $v$ at points $(x_\e, t_\e)$ and $(y_\e, s_\e)$, respectively, with
\begin{align*}
\sign \psi_u(\cdot + x_\e') \leq \sign \psi_v(\cdot + y_\e'),
\end{align*}
and there is $\delta > 0$ with $B_{\delta_\e}(0) \subset \set{\psi_u(\cdot + x_\e') = 0} \cap \set{\psi_v(\cdot + y_\e') = 0}$.
Here $x_\e' := \TT_\e^* x_\e$ and so on as before.
Then Lemma~\ref{le:lambda-const-dep} and Corollary~\ref{co:constant-diff-lambda} yield
\begin{align*}
\essinf &\bra{\Lambda_{p_\e, f_{u}(\cdot+ x_\e')}[\psi_u(\cdot + x_\e')]} \leq
\esssup \bra{\Lambda_{p_\e, f_{u}(\cdot + x_\e')}[\psi_v(\cdot + y_\e')]} \\
&\leq
\esssup \bra{\Lambda_{p_\e, f_{v}(\cdot +y_\e')}[\psi_v(\cdot + y_\e')]} + L_f|x_\e - y_\e| + \omega_f(|t_\e - s_\e|),
\end{align*}
where $f_{u}(w) := f(\TT w + \TT_\perp x_\e'', t_\e)$, $f_v(w) := f(\TT w + \TT_\perp y_\e'', s_\e)$ and $L_f$ is the Lipschitz constant of $f$ in space and $\omega_f$ is the modulus of continuity of $f$ on a sufficiently large bounded set, and $\essinf$ and $\esssup$ are taken over $B_{\delta}(0)$. Since $\delta$ can be taken arbitrarily small, we deduce
\begin{align}
\label{l-order}
\underline\Lambda_{p_\e, f_{u}}[\psi_u](x_\e') \leq
\overline\Lambda_{p_\e, f_{v}}[\psi_v](y_\e') + L_f|x_\e - y_\e| + \omega_f(|t_\e - s_\e|).
\end{align}

Then we have from the definition of viscosity subsolution and supersolution
\begin{align}
\label{tf-ineq}
\begin{aligned}
\frac{\mu}{(T-t_\e)^2} &+ F(x_\e, t_\e, p_\e, \underline\Lambda_{p_\e, f_{u}}[\psi_u](x_\e')) \\
&- F(y_\e, s_\e, p_\e, \overline\Lambda_{p_\e, f_{v}}[\psi_v](y_\e')) \leq 0.
\end{aligned}
\end{align}
On the other hand, using \eqref{l-order} and the monotonicity of $F$ in the last variable, we can estimate
\begin{align*}
F(x_\e, &t_\e, p_\e, \underline\Lambda_{p_\e, f_{u}}[\psi_u](x_\e'))
- F(y_\e, s_\e, p_\e, \overline\Lambda_{p_\e, f_{v}}[\psi_v](y_\e'))\\
&\geq F\big(x_\e, t_\e, p_\e, \overline\Lambda_{p_\e, f_{v}}[\psi_v](y_\e')  + L_f|x_\e - y_\e| + \omega_f(|t_\e - s_\e|)\big)\\
&\quad - F(y_\e, s_\e, p_\e, \overline\Lambda_{p_\e, f_{v}}[\psi_v](y_\e'))\\
&=: I.
\end{align*}

If we assume (\ref{v-Lipschitz}), that is, $v$ is Lipschitz in space, uniformly in time, with Lipschitz constant $L_v$, we must have $|p_\e| \leq L_v$. Therefore we can find a modulus of continuity $\omega_F$ of $F$ on a sufficiently large bounded subset of $\R^n \times \R \times \Rn \times \R$ and we can estimate
\begin{align*}
I \geq - \omega_F\big(|x_\e - y_\e| +|t_\e - s_\e| + L_f|x_\e - y_\e| + \omega_f(|t_\e - s_\e|)\big),
\end{align*}
where the right-hand side converges to 0 as $\e \to 0$.

On the other hand, if we assume (\ref{f-Lipschitz}), that is, that $F$ satisfies \eqref{F-regularity} and $f$ is Lipschitz in both variables, we can estimate
\begin{align}
\label{final-comp-ineq}
\begin{aligned}
I &\geq -C (L_f+1) (|p_\e| + 1)(|x_\e - y_\e| + |t_\e - s_\e|)\\
&\geq -K (|p_\e| + 1) \sqrt\e,
\end{aligned}
\end{align}
for some $K$ independent of $\e$, where we used \eqref{point-max-dist}.
But $|p_\e| \sqrt{\e} \to 0$ as $\e \to 0$ by \eqref{p-limit} and so the right-hand side converges to 0 as $\e \to 0$.

Either way, since by \eqref{tf-ineq} we have
\begin{align*}
I \leq - \frac\mu{T^2} < 0,
\end{align*}
we arrive at a contradiction.  We conclude that $\sup u - v \leq 0$. This finishes the proof of the comparison principle.
\end{proof}

\section{Stability}
\label{sec:stability}

The stability with respect to an approximation by solutions of regularized problems
\begin{align}
\label{regularized-problem}
u_t + F\big(t, \nabla u, \divo(\nabla \sigma_m(\nabla u)) - f\big) = 0
\end{align}
is an extension of the proof in \cite{GP_ADE}.
Recall that we consider two modes of regularization of $\sigma$:
\begin{enumerate}
\item $\sigma_m\in C^2(\Rn)$, $a_m^{-1} \leq \nabla^2 \sigma_m \leq a_m$ for some $a_m > 0$, and $\sigma_m \searrow \sigma$; or
\label{reg-parabolic}
\item $\sigma_m$ is an anisotropy, $\sigma_m \in C^2(\R^n \setminus \set0)$, $\set{\sigma_m < 1}$ is strictly convex, and $\sigma_m \to \sigma$ locally uniformly.
\label{reg-linear-growth}
\end{enumerate}

We have the following theorem for the approximation scheme \ref{reg-parabolic}.

\begin{theorem}
\label{th:stability}
Let $\sigma$ be a crystalline anisotropy, $F \in C(\R \times \R^n \times \R)$ be non-increasing in the last variable, and $f \in C(\Rn \times \R)$ be Lipschitz continuous in space, uniformly in time.
Let $\set{u_m}$ be a locally bounded sequence of viscosity subsolutions of \eqref{regularized-problem} on an open set $U \subset \R^n\times \R$ with $\sigma_m$ as in \ref{reg-parabolic}. Then ${\limsup^*}_{m\to\infty} u_m$ is a viscosity subsolution of \eqref{pde} on $U$. Similarly, if $u_m$ are viscosity supersolutions then ${\liminf_*}_{m\to\infty} u_m$ is a viscosity supersolution of \eqref{pde}.
\end{theorem}

\begin{remark}
We can also allow for $f$ to be locally uniformly approximated by a sequence $f_m$ with a uniform Lipschitz constant in space, and for $F$ to be locally uniformly approximated by continuous functions $F_m$, nonincreasing in the last variable. This will only figure in \eqref{first-visc-ineq} where we need to add a subscript $m_l$ to $F$ and $f$. Due to the locally uniform convergence we recover \eqref{second-visc-ineq}.
\end{remark}

\begin{remark}
Note that unlike in the comparison principle, Theorem~\ref{th:comparison-principle}, we cannot allow $F$ to explicitly depend on the space variable $x$. Indeed, in the proof below, we only know that the accumulation points of $\set{x_a}_{a >0}$ lie in the set $N \ni \hat x$, and it is not clear that $\hat x$ is one of them to deduce \eqref{last-visc-cond}. Note that a typical perturbation like adding $|x|^4$ to $\varphi$ does not work since this does not yield an admissible faceted test function.
\end{remark}

For the approximating scheme \ref{reg-linear-growth} by anisotropies $\sigma_m$ we need to assume that the approximating sequence $\set{u_m}$ are solutions with uniformly bounded continuous initial data.

\begin{theorem}
\label{th:stability-linear-growth}
Let $\sigma$, $F$ and $f$ be as in Theorem~\ref{th:stability}.
Let $T > 0$ and let $u_m$ be a locally bounded sequence of viscosity solutions of \eqref{regularized-problem} on $\R^n \times (0, T)$ with $\sigma_m$ as in \ref{reg-linear-growth} with initial data $u_m(\cdot, 0) = u_{0,m}$, where $u_{0, m} \in C(\R^n)$ are uniformly bounded. Then ${\limsup^*}_{m\to\infty} u_m$ is a viscosity subsolution of \eqref{pde} and ${\liminf_*}_{m\to\infty} u_m$ is a viscosity supersolution of \eqref{pde} on $\R^n \times (0, T)$.
\end{theorem}

\begin{proof}
As in \cite{GP_ADE}, we approximate each $u_m$ by solutions $u_{m,\delta}$ of \eqref{regularized-problem} with initial data $u_{0, m}$ and $\sigma_{m, \delta}$ smooth with quadratic growth, as in \ref{reg-parabolic} in the limit $\delta \searrow 0$. The rest follows the proof of Theorem~\ref{th:stability}. We refer the reader to \cite{GP_ADE} for further details.
\end{proof}

\subsection{Proof of Theorem~\ref{th:stability}}
Let
\begin{align*}
u = {\limsup_{m \to \infty}}^* u_m.
\end{align*}
We follow the approach in \cite{GP_CPAM}. We start with the simpler setting when $\hat p = 0$ and $f = f(x)$ to explain the structure of the proof.

\medskip
Suppose that a $0$-admissible test function is of the form $\varphi(x,t) = \psi(x) + g(t)$ where $\psi$ is $0$-admissible with facet containing $0$ in its interior, and suppose that $u - \varphi(\cdot - h, \cdot)$ has a global maximum $0$ at $(x,t) = (0,0)$ for all $|h| \leq \rho$ for some fixed $\rho > 0$. By Remark~\ref{rem:periodic-extension}, we can modify $\psi$ away from the facet $\set{\psi = 0}$ by distance greater than $2\rho$ so that $\psi$ is a periodic Lipschitz function with a Cahn-Hoffman vector field. We also modify $f$ away from the facet so that it is periodic and Lipschitz. By rescaling, we can assume that $\psi$ is $\Z^n$ periodic so that $\psi \in Lip(\T^n)$. We have $\partial \E_f(\psi)\neq\emptyset$ where $\E_f$ is defined as in \eqref{Ef} over $\Tn$.
Since $\Lambda_{0, f}[\psi] = -\partial^0 \E_f(\psi)$ on $\set{\psi = 0}$, we can approximate it by the resolvent problems for $\E_f$ and $\E_{m,f}$,
\begin{align*}
\psi_a + a &\partial \E_f(\psi_a) \ni \psi, &
\psi_{a,m} + a &\partial \E_{m,f}(\psi_{a,m}) \ni \psi.
\end{align*}
Since the resolvent problem for $\E_{m,f}$ is uniformly elliptic, standard regularity yields $\psi_{a, m} \in C^2(\Tn)$. Moreover, $\psi_a$ and $\psi_{a,m}$ are Lipschitz continuous with Lipschitz constant $L + a L_f$, where $L$ is the Lipschitz constant of $\psi$ and $L_f$ is the Lipschitz constant of $f$; see Section~\ref{sec:resolvent-problem}.
Due to the Mosco convergence of $\E_{m,f}$ to $\E_f$, we deduce the convergence $\psi_{a,m} \to \psi_a$ in $L^2$. Therefore the convergence is uniform due to the Lipschitz continuity.
Finally, $(\psi_a - \psi)/a \to - \partial^0 \E_f(\psi)$ in $L^2(\T^n)$.

We define the $\rho$-neighborhood of the facet,
\begin{align*}
O := \set{x : \psi(y) = 0 \text{ for some } y, |y - x| \leq \rho}.
\end{align*}
Note that we did not modify $\psi$ above closer than $2\rho$ from the facet $\set{\psi = 0}$ and so $\psi(\cdot - w)$ was not modified on $O$ for all $|w| \leq \rho$. We define
\begin{align*}
\bar u(x) := \sup_{|t| \leq \rho} [u(x, t) - g(t)].
\end{align*}
We have $\bar u - \psi(\cdot - w)$ on $O$ for $|w| \leq \rho$, with equality at $x = 0$.

For $\delta := \rho/5$, we define the set
\begin{align*}
N := \set{x \in O: \bar u(x) \geq 0,\ \psi(x - w) \leq 0 \text{ for some } |w| \leq \delta}.
\end{align*}
Note that $N$ is bounded since $O$ is bounded.
Since $\bar u - \psi(x - z)$ has a maximum $0$ for any $|z| \leq \rho = 5\delta$, we showed in
\cite[Corollary~8.3]{GP_ADE} that
\begin{align*}
\bar u(x) \leq 0 \quad \text{and} \quad \psi(x - z) \geq 0 \qquad \text{for all } \dist(x, N) \leq 3\delta,\ |z| \leq \delta,
\end{align*}
and $\dist(N, \partial O) \geq 4\delta$.

The key consequence is that we can modify $\psi$ away from the facet while preserving its sign and the fact that all points of maxima of $\bar u - \psi(\cdot - z)$ in $3\delta$-neighborhood of $N$ all lie in $N$, whenever $|z| \leq \delta$. In particular, we can assume that $\psi$ has arbitrarily small Lipschitz constant $L > 0$ by simply multiplying $\psi$ by a small positive number. It is convenient to introduce $M^s$, the $s$-neighborhood of $N \times \set{0}$, as
\begin{align*}
M^s := \set{(x, t) : \dist(x, N) \leq s,\ |t| \leq s}.
\end{align*}
By adding $|t|^2$ to $g(t)$, we may also assume that $u - \varphi(\cdot - z, \cdot)$ takes on its maximum $0$ in the set $M^{3\delta}$ at $t = 0$, whenever $|z| \leq \delta$.

To be able to conclude \eqref{res-min} below, for every $a > 0$ we choose $z_a$ with $|z_a| \leq \delta$ satisfying $\psi(z_a) = \min_{|w| \leq \delta} \psi_a(w)$.
By the uniform convergence of $\psi_a \to \psi$, for all $a$ small enough all the maxima of $u - \varphi_a(\cdot + z_a, \cdot)$ in $M^{3\delta}$ lie in $M^\delta$, that is, $\argmax_{M^{3\delta}}[u - \varphi_a(\cdot + z_a, \cdot)] \subset M^\delta$. For any such small $a$ there exists a point of maximum $(x_a, t_a)$ in $M^\delta$ and a sequence $m_\ell \to \infty$ and a sequence $(x_\ell, t_\ell) \to (x_a, t_a)$ of points of maxima of $u_{m_\ell} - \varphi_{a, m_\ell}(\cdot + z_a, \cdot)$ in $M^{3\delta}$. Since $\psi_{a, m}$ are Lipschitz continuous with Lipschitz constant $L+aL_f$, we can also assume that $\nabla \psi_{a, m_\ell} \to p_a$ as $\ell \to \infty$ for some $|p_a| \leq L + aL_f$.

The definition of the viscosity subsolution $u_m$ implies
\begin{align}
\label{first-visc-ineq}
g'(t_\ell) + F\Big(t_\ell, \nabla \psi_{a, m_\ell}(x_\ell + z_a), \divo \big(\nabla\sigma_{m_\ell}(\nabla \psi_{a, m_\ell})\big)(x_\ell + z_a) - f(x_\ell)\Big) \leq 0.
\end{align}
But $\divo \big(\nabla\sigma_{m_\ell}(\nabla \psi_{a, m_\ell})\big) - f = -\partial^0 \E_{m_\ell, f}(\psi_{a, m_\ell}) = (\psi_{a, m_\ell} - \psi)/ a$, and so by the uniform convergence $\psi_{a, m_\ell} \to \psi_a$ we recover in the limit $\ell \to \infty$ the inequality
\begin{align}
\label{second-visc-ineq}
g'(t_a) + F\pth{t_a, p_a, \frac{\psi_a(x_a + z_a) - \psi(x_a + z_a)}a + f(x_a + z_a) - f(x_a)} \leq 0.
\end{align}
Due to the choice of $z_a$ and \cite[Lemma~8.5]{GP_ADE}, we have
\begin{align}
\label{res-min}
\frac{\psi_a - \psi}a (x_a + z_a) \leq \min_{|w|\leq \delta} \frac{\psi_a - \psi}a(w).
\end{align}
The monotonicity of $F$ in the last variable, uniform continuity of $F$ on compact sets and the $L_f$-Lipschitz continuity of $f$ yields
\begin{align*}
g'(t_a) + F\pth{t_a, p_a, \min_{|w|\leq \delta} \frac{\psi_a - \psi}a(w)} \leq \omega_F(L_f |z_a|),
\end{align*}
for a modulus of continuity $\omega_F$.
We now find a sequence $a_j \to 0$ such that $p_{a_j} \to p$ for some $|p| \leq L$, and
\begin{align*}
\lim_{j\to \infty} \min_{|w|\leq \delta} \frac{\psi_{a_j} - \psi}{a_j}(w) = \liminf_{a \searrow 0}
\min_{|w|\leq \delta} \frac{\psi_a - \psi}a(w).
\end{align*}
Since $t_{a_j} \to 0 = \hat t$, we recover
\begin{align*}
g'(0) + F\pth{0, p, \liminf_{a \searrow 0} \min_{|w|\leq \delta} \frac{\psi_a - \psi}a(w)} \leq \omega_F(L_f \delta).
\end{align*}
And since $(\psi_a - \psi)/a \to \Lambda_{0,f}[\psi]$ in $L^2(B_\delta(0))$, we have by the monotonicity of $F$ in the last variable
\begin{align*}
g'(0) + F(0, p, \essinf_{|w|\leq \delta} \Lambda_{0,f}[\psi](w)) \leq \omega_F(L_f\delta).
\end{align*}
As we explained above, $L > 0$ can be taken arbitrarily small, $|p| \leq L$ and so by continuity
\begin{align}
\label{last-visc-cond}
g'(0) + F\Big(0, 0, \essinf_{|w|\leq \delta} \Lambda_{0,f}[\psi](w)\Big) \leq \omega_F(L_f\delta).
\end{align}

Since we can take $\rho > 0$ and therefore $\delta > 0$ arbitrarily small, we deduce \eqref{visc_subsolution}.  We have verified \eqref{visc_subsolution} for $\hat p = 0$.

\bigskip

In the case of general $\hat p$, and $f = f(x,t)$, we have a more complicated $\hat p$-admissible test function $\varphi(x,t) = \psi(x') + \theta(x'') + \hat p \cdot x + g(t)$. Let $k$, $Z$, $\TT$ and $\TT_\perp$ correspond to $\hat p$. Suppose that $u - \varphi(\cdot - \TT w, \cdot)$ has a global max at $(\hat x, \hat t)$ for small $w \in \R^k$. By translation, we may assume that $(\hat x, \hat t) = (0, 0)$. By rotating the coordinate system, we may assume that $Z= \R^k \times \set0$, $Z^\perp = \set0 \times \R^{n-k}$ so that $x = (x', x'')$, $x' \in \R^k$, $x'' \in \R^{n-k}$.

We proceed as in \cite{GP_ADE}. First, recalling \cite[Corollary~8.3]{GP_ADE}, we note that by adding $|x''|^2$ to $\theta(x'')$ and $|t|^2$ to $g(t)$, we may assume that the maximum of $u - \varphi(\cdot - \TT w, \cdot)$ in $M^{3\delta}$ is attained only on $M^0 = N \times \set0 \times \set0$ for any small $w$, where
\begin{align*}
M^s := \set{(x, t): \dist(x', N) \leq s,\ |x''| \leq s, \ |t| \leq s}.
\end{align*}
By the above mentioned corollary and $\nabla \theta(0) = 0$, we may also assume that the Lipschitz constant of $\psi$ and $\theta$ on $M^{3\delta}$ is smaller than any chosen fixed $L > 0$.
By modifying $\psi$ away from $N$ and $\theta$ away from $0$, we can make them periodic, while preserving the Lipschitz constant $L$. As $N$ is bounded, scaling allows us to assume that they are $1$ periodic and that $\diam(N) < 1$. This way we can again consider a resolvent problem on $\T^n$, but with the reduced energy
\begin{align*}
\bar\E(v) := \int \bar\sigma(Dv) + \bar f v, \qquad v \in BV(\Tn) \cap L^2(\Tn),
\end{align*}
where
\begin{align*}
\bar\sigma(p) = \lim_{\lambda \searrow 0} \frac{\sigma(\hat p + \lambda p) - \sigma(\hat p)}\lambda,
\end{align*}
and
\begin{align*}
\bar f(x) := f(\TT x', 0), \qquad \dist(x', \set{\psi = 0}) \text{ small},
\end{align*}
and $\bar f$ is extended periodically away from the facet in the directions in $Z$. Note that $\bar\sigma$ is linear on $Z^\perp$, and $\bar f$ is constant in the directions in $Z^\perp$. In fact, $\TT x' = \TT \TT^* x$ is the orthogonal projection on $Z$.

As we showed in \cite[Lemma~3.9]{GP_ADE}, the solution $\bar \psi_a$ of the resolvent problem
\begin{align*}
\bar \psi_a + a \partial \bar\E(\bar \psi_a) \ni \bar \psi
\end{align*}
with $\bar \psi(x) = \psi(x') + \theta(x'')$ is of the form
\begin{align*}
\bar \psi_a(x) = \psi_a(x') + \theta(x''),
\end{align*}
where $\psi_a$ is the solution of the sliced resolvent problem
\begin{align*}
\psi_a + a \partial \SE(\psi_a) \ni \psi
\end{align*}
with the sliced energy
\begin{align*}
\SE(v) := \int \sigma^{\rm sl}_{\hat p}(D v) + \hat f v, \qquad v \in BV(\T^k) \cap L^2(\T^k),
\end{align*}
with $\hat f(w) := \bar f(\TT w)$.
We recall that
\begin{align*}
\frac{\psi_a - \psi}a \to \Lambda_{\hat p, \hat f}  \qquad \text{as $a \to 0$ in } L^2(\set{\psi = 0}).
\end{align*}

Let $\bar\psi_{a,m}$ be the solution of the resolvent problem
\begin{align*}
\bar \psi_{a,m} + a \partial \bar\E_m(\bar \psi_{a,m}) \ni \bar \psi,
\end{align*}
where
\begin{align*}
\bar\E_m (v) := \int \sigma_m(\nabla v + \hat p) + \bar f v, \qquad v \in H^1(\Tn).
\end{align*}
We recall that $\bar\E_m$ Mosco-converges to the energy $\int \sigma(Dv +\hat p) + \bar f v$, which is equal to $\bar \E$ for functions with sufficiently small Lipschitz constant since $\sigma$ and $\bar \sigma$ differ by a constant on a small neighborhood of $\hat p$ depending only on $\hat p$ since $\sigma$ is crystalline. Thus we have $\bar \psi_{a,m} \rightrightarrows \bar \psi_a$ as long as $L, a > 0$ are taken sufficiently small. For more details, see \cite{GP_ADE}.

For each $a > 0$, we select $z_a \in \R^k$ with
\begin{align*}
\psi_a(z_a) = \min_{|w| \leq \delta} \psi_a(w).
\end{align*}
For a fixed $a > 0$, there exist a sequence of points of maximum $(x_{a,m}, t_{a,m})$ of $u_m(x,t) - \bar\psi_{a,m}(x + \TT z_a) - \hat p \cdot x - g(t)$ on $M^{3\delta}$ and a point of maximum $(x_a, t_a)$ of $u(x,t) - \psi_a(x' + z_a) - \theta(x'') - \hat p \cdot x - g(t)$ on $M^{3\delta}$ such that $(x_{a,m}, t_{a,m}) \to (x_a, t_a)$ as $m \to \infty$ along a subsequence.

By the viscosity subsolution condition for \eqref{regularized-problem} we have
\begin{align*}
g'(t_{a,m}) + F\Big(&t_{a,m}, \nabla \bar \psi_{a, m}(x_a) + \hat p,\\
&\quad\divo\nabla\sigma_m(\nabla\bar\psi_{a,m} + \hat p) (x_{a,m} + \TT z_a) - f(x_{a,m}, t_{a,m})\Big) \leq 0.
\end{align*}
By the Lipschitz bound for $\psi_{a,m}$, there is $p_a$ such that $\nabla \bar \psi_{a, m}(x_{a,m}) \to p_a$ along a subsequence $m \to \infty$. We also have $|p_a| \leq L + aL_f$. Since $\divo\nabla\sigma_m(\nabla\bar\psi_{a,m} + \hat p) = (\bar\psi_{a,m} - \bar \psi)/a + \bar f$, continuity and uniform convergence yield along the subsequence $m \to \infty$ that
\begin{align*}
g'(t_a) + F\Big(&t_a, p_a + \hat p, \\
&\quad (\psi_a(x_a' + z_a) - \psi(x_a' + z_a))/ a + \bar f(x_a + \TT z_a) - f(x_a, t_a)\Big) \leq 0.
\end{align*}

Now we estimate
\begin{align*}
\bar f(x_a + \TT z_a) - f(x_a, t_a) &= f(\TT(x_a' + z_a), 0) - f(\TT x_a' + \TT_\perp x_a'', t_a) \\
&\leq \omega_f(|x_a''| + |t_a| + |z_a|),
\end{align*}
where $\omega_f$ is a modulus of continuity of $f$ on a sufficiently large subset of $\Rn \times \R$.
We use the above estimate and \cite[Lemma~8.5]{GP_ADE} to deduce
\begin{align*}
g'(t_a) + F\Big(t_a, p_a + \hat p, \min_{|w| \leq \delta}(\psi_a(w) - \psi(w))/ a\Big) \leq \omega_F(\omega_f(|x_a''| + |t_a| + \delta)),
\end{align*}
where $\omega_F$ is a modulus of continuity of $F$ on a sufficiently large subset of $\R \times \Rn \times \R$.
Now sending $a \searrow 0$ along a subsequence so that $p_a \to p$ for some $p$ with $|p| \leq L$ and using that $(\psi_a - \psi)/a \to \Lambda_{\hat p, \hat f}[\psi]$ in $L^2(B_\delta(0))$, $|x_a''| \to 0$ and $|t_a| \to 0$, we recover
\begin{align*}
g'(0) + F(0, p + \hp, \essinf_{|w|\leq \delta} \Lambda_{\hp, \hat f}[\psi](w)) \leq \omega_F(\omega_f(\delta)).
\end{align*}
Since we can take $L > 0$ and $\delta > 0$ as small as we want, we deduce
\begin{align*}
g'(0) + F(0, \hp, \underline\Lambda_{\hp, \hat f}[\psi](0)) \leq 0.
\end{align*}
Therefore the viscosity solution condition \eqref{visc_subsolution} is satisfied for any $\hp$-admissible test function.

This shows that $u$ is a viscosity subsolution of \eqref{pde}. Since the proof that ${\liminf_*}_{m\to\infty} u_m$ is a viscosity supersolution of \eqref{pde} is analogous, this finishes the proof of Theorem~\ref{th:stability}.

\section{Lipschitz bound}
\label{sec:lipschitz-bound}

In this section we show that the solutions of the approximating problems \eqref{regularized-problem} have a modulus continuity uniform in $m$.
We consider a level set equation of $V=g(\nu,\kappa_\sigma+f)$ of the form
\begin{equation} \label{E}
	u_t + |\nabla u| g \big( \nabla u/|\nabla u|, -\operatorname{div}\left(\nabla\sigma(\nabla u)\right) + f(x,t) \big) = 0
	\quad\text{in}\quad \mathbb{R}^n \times(0,T),
	\quad T>0,
\end{equation}
where the anisotropy $\sigma$ is smooth. For the following Lipschitz estimate we do not need to assume that $\sigma$ is positively one-homogeneous.
\begin{theorem} \label{L}
Assume that $\sigma \in C^2(\Rn \setminus \set0)$ is convex, $g \in C(\mathcal S^{n-1} \times \R)$ is Lipschitz in the second variable with a Lipschitz constant $L_g$ uniform in the first variable, $g$ is non-decreasing in the second variable, and $f \in C(\Rn \times[0,T])$ is Lipschitz in space with a Lipschitz constant $L_f$ uniform in time.
 Let $u \in C\left( \mathbb{R}^n \times [0,T] \right)$ be a solution of \eqref{E} in $\mathbb{R}^n \times (0,T)$.
 Assume that $u$ is constant outside $B \times [0,T]$, where $B$ is a (large) ball.
 If the initial data $u_0$ is Lipschitz, then
\begin{align}
\label{u-Lip-space}
	\left| u(x,t) - u(y,t) \right| \leq Le^{Mt} |x-y|, \quad
	x,y \in \mathbb{R}^n, \quad
	t \in [0,T]
\end{align}
with $M=L_g L_f$, $L=\operatorname{Lip}(u_0)$.
\end{theorem}
\begin{proof}
Suppose that the conclusion were false.
 Then,
\[
	m_0 := \sup_{\substack{Q \times Q \\ t=s}} \Phi (x,t,y,s)>0, \quad
	\Phi(x,t,y,s) := u(x,t)-u(y,s) - Le^{Mt}|x-y|,
\]
where $Q=\mathbb{R}^n\times[0,T]$.
 By uniform continuity of $u$ in $Q$, for any $\varepsilon>0$ there is $\delta>0$ such that $|t-s|\leq\delta$ implies that $\left|u(x,t)-u(x,s)\right|<\varepsilon$ for all $x\in\mathbb{R}^n$.
 We fix $\varepsilon=m_0/8$ so that
\[
	\sup_{\substack{Q \times Q \\ |t-s|\leq\delta}} \Phi(x,t,y,s) \geq m_0/2.
\]
We take $\beta>0$ sufficiently small and fix it so that
\[
	m_1 := \sup_{\substack{Q \times Q \\ |t-s|\leq\delta}} \left( \Phi(x,t,y,s) - \beta/(T-t) \right) \geq m_0/4.
\]
For $\alpha>1$, we consider
\[
	\Psi (x,t,y,s) := \Phi (x,t,y,s) - \beta/(T-t) - \alpha(t-s)^2.
\]
Since $u$ is continuous and is a constant on $(\Rn\setminus B)\times[0,T]$ with a big ball $B$, there is a maximizer $(x_\alpha, t_\alpha, y_\alpha, s_\alpha) \in Q \times Q$ of $\Psi$.
 Moreover, if $\alpha$ is sufficiently large, say $\alpha>\alpha_0$, then $|t_\alpha - s_\alpha|<\delta$.
 This yields
\[
	m_1 \geq \max_{Q \times Q} \Psi(x,t,y,s), \quad
	c_0 := \sup_{\alpha>\alpha_0} 1/|x_\alpha - y_\alpha| < \infty
\]
by uniform continuity of $u$ in $Q$.
 Moreover, we see that $t_\alpha - s_\alpha \to 0$ as $\alpha\to\infty$ since $m_1>0$.
 Since $L=\operatorname{Lip}u_0$, we observe that $\Phi(x,0,y,0)\leq 0$ for all $x,y\in\mathbb{R}^n$.
 Thus any accumulation point of $\{t_\alpha\}$ should not be zero since $t_\alpha-s_\alpha \to 0$.
 We may assume that $t_\alpha, s_\alpha$ are away from zero for sufficiently large $\alpha$, say $\alpha>\alpha_1\geq\alpha_0$.
 Evidently, $t_\alpha, s_\alpha < T$.
 We rewrite $\Psi$ as
\begin{multline*}
	\Psi(x,t,y,s) = u(x,t) - u(y,s) - \varphi(x,t,y,s), \\
	\varphi(x,t,y,s) := \alpha(t-s)^2 + \beta/(T-t) + Le^{Mt} |x-y|.
\end{multline*}
Our maximizer $z_\alpha = (x_\alpha,t_\alpha,y_\alpha,s_\alpha)$ should be in the interior of $Q \times Q$.

We are now in position to apply a parabolic version of Crandall--Ishii's lemma \cite{CIL}, \cite[Theorem 3.3.3]{GigaBook}.
 It provides $n \times n$ symmetric matrices $X_\alpha, Y_\alpha$ such that
\[
	\left( \varphi_t(z_\alpha), \nabla_x \varphi (z_\alpha), X_\alpha \right), \quad
	\left(\text{resp.}\ (-\varphi_s(z_\alpha), -\nabla_y\varphi(x_\alpha), -Y_\alpha) \right)
\]
can be approximated by super (resp.\ sub) parabolic semijets of $u$ at $(x_\alpha \ t_\alpha)$ (resp.\ $(y_\alpha,s_\alpha)$).
 Moreover, $X_\alpha + Y_\alpha \leq 0$ and the norms $\|X_\alpha\|, \|Y_\alpha\|$ are bounded by the spatial second derivatives of $\varphi$ at $z_\alpha$ for $\alpha>\alpha_1$ which are bounded since $c_0<\infty$.
 Since $u$ is a sub- and supersolution and $\nabla_x \varphi(z_\alpha) = -\nabla_y \varphi(z_\alpha) =: p_\alpha$, we end up with
\begin{equation} \label{IN1}
	\varphi_t + |p_\alpha| g \left( p_\alpha/|p_\alpha|,
	-\operatorname{tr}\left(\nabla^2 \sigma(p_\alpha)X_\alpha\right)
	+ f(x_\alpha, t_\alpha) \right) \leq 0 \quad\text{at}\quad z_\alpha
\end{equation}
\[
		- \varphi_s + |p_\alpha| g \left( p_\alpha/|p_\alpha|,
	-\operatorname{tr}\left(\nabla^2 \sigma(p_\alpha)(-Y_\alpha)\right)
	+ f(y_\alpha, s_\alpha) \right) \geq 0 \quad\text{at}\quad z_\alpha.  \]
In the second inequality, we may replace $-Y_\alpha$ by $X_\alpha$ since $X_\alpha+Y_\alpha\leq 0$ and $g$ is monotone in the last variable (ellipticity).
 The resulting inequality is
\begin{equation} \label{IN2}
		- \varphi_s + |p_\alpha| g \left( p_\alpha/|p_\alpha|,
	-\operatorname{tr}\left(\nabla^2 \sigma(p_\alpha)X_\alpha\right)
	+ f(y_\alpha, s_\alpha) \right) \geq 0 \quad\text{at}\quad z_\alpha.
\end{equation}
Since
\[
	 \left| g(p,\xi_1) - g(p,\xi_2) \right| \leq L_g |\xi_1-\xi_2|,
\]
subtracting \eqref{IN2} from \eqref{IN1} yields
\begin{equation} \label{IN3}
		\varphi_t(z_\alpha) + \varphi_s(z_\alpha) - |p_\alpha| L_g
		\left| f(x_\alpha\ t_\alpha) - f(y_\alpha, s_\alpha) \right| \leq 0.
\end{equation}
By a simple manipulation, we get
\begin{align*}
	\varphi_t (z_\alpha) &= \beta/(T-t_\alpha)^2 + LMe^{Mt_\alpha}|x_\alpha - y_\alpha| + 2\alpha(t_\alpha-s_\alpha), \\
	\varphi_s (z_\alpha) &= -2\alpha(t_\alpha-s_\alpha), \\
	p_\alpha &= Le^{Mt_\alpha}(x_\alpha-y_\alpha)/|x_\alpha-y_\alpha|\ \text{ so that }\ |p_\alpha|=Le^{Mt_\alpha}.
\end{align*}
Since $f$ is uniformly continuous in $B \times [0,T]$ and spatially Lipschitz uniformly in time,
\[
	\left| f(x_\alpha,t_\alpha) - f(y_\alpha,s_\alpha) \right|
	\leq L_f |x_\alpha - y_\alpha| + \omega(t_\alpha-s_\alpha),
\]
where $\omega$ is a modulus, i.e., $\omega \in C[0,\delta]$, $\omega\geq 0$ and $\omega(0)=0$.
 Since $1/(T-t_\alpha)^2 \leq 1/T^2$, \eqref{IN3} now yields
\begin{equation} \label{IN4}
		\beta/T^2 + LMe^{Mt_\alpha} |x_\alpha-y_\alpha|
		- |p_\alpha|L_g \left( L_f |x_\alpha-y_\alpha| + \omega(t_\alpha-s_\alpha) \right) \leq 0.
\end{equation}
Since $M=L_g L_f$, we see that
\[
	LMe^{Mt_\alpha} |x_\alpha-y_\alpha| - |p_\alpha| L_g L_f |x_\alpha-y_\alpha| = 0.
\]
Thus \eqref{IN4} implies
\[
	\beta/T^2 - |p_\alpha| L_g \omega(t_\alpha-s_\alpha) \leq 0.
\]
Sending $\alpha\to\infty$ and using $t_\alpha-s_\alpha \to 0$, we obtain $\beta/T^2 \leq 0$, which is a contradiction.
 The proof is now complete.
\end{proof}

As an application of Lipschitz bound, we further derive a uniform H\"older continuity in time when $\sigma$ is positively one-homogeneous.
\begin{theorem} \label{H}
Assume that $\sigma$, $g$ an $f$ are as in Theorem~\ref{L} and additionally assume that $\sigma$ is positively one-homogeneous.
 Let $u$ be the solution of \eqref{E} in $\mathbb{R}^n \times (0,T)$.
 Assume the same hypotheses of Theorem \ref{L} concerning $u$.
 Assume that $f$ is bounded. Then
\begin{align}
\label{u-Holder-time}
	\left| u(x,t) - u(x,s) \right| \leq A|t-s|^{1/2}, \quad
	x \in \mathbb{R}^n, \quad t,s \in [0,T]
\end{align}
with some constant $A$ depending only on positive constants $L$, $M$, $T$, $G$, $\lambda_+$, $\lambda_-$ such that
\begin{multline*}
	\sup_{r>0}\sup_{p\in \mathcal S^{n-1}} r \left| g\left( p, \pm(n-1)/r+\sup|f| \right) \right| \leq G, \\
	\lambda_-|x| \leq \sigma(x) \leq \lambda_+|x| \quad\text{for all}\quad x \in \mathbb{R}^n.
\end{multline*}
\end{theorem}
\begin{proof}
We shall prove this Lemma by constructing a barrier.
 We first observe that
\[
	\nabla\sigma \left(\nabla\sigma^\circ(x)\right) = x/|x|
\]
so that $\operatorname{div}\nabla\sigma\left(\nabla\sigma^\circ(x)\right)=(n-1)/|x|$, where $\sigma^\circ(x)$ is the support function of the set $\left\{ \sigma(x)\leq 1 \right\}$, i.e.,
\[
	\sigma^\circ(x) = \sup\left\{ \langle x,y \rangle \mid \sigma(y) \leq 1 \right\}.
\]
The function $\sigma^\circ$ satisfies
\[
	\lambda^{-1}_+ |x| \leq \sigma^\circ(x) \leq \lambda^{-1}_-|x|
\]
and it is convex and positively one-homogeneous.
 (If $\nabla^2 \sigma$ is strictly positive on $\mathcal S^{n-1}$, this $\sigma^\circ$ may not be smooth but in the sense of viscosity solutions the indentity $\operatorname{div}\left(\nabla\sigma(\nabla\sigma^\circ(x)\right)=(n-1)/|x|$ still holds.)

For a given $s \in [0,T]$, by Theorem \ref{L}, $u(\cdot, s)$ is Lipschitz with a Lipschitz constant $L_s = Le^{Ms}$.
 For a given $\delta>0$ by Young's inequality, we have
\[
	|x| \leq \delta + |x|^2 /4\delta.
\]
This implies that
\[
	\left(u(x,s) - u(x_0,s)\right) / L_s \leq \delta + |x-x_0|^2/4\delta
	\leq \delta + \lambda^2_+ \sigma^\circ(x-x_0)^2/4\delta.
\]
If we take a constant $C$ large but depending only on $G$, we observe that
\[
	v(x,t) := L_s \left(C(t-s)/\delta + \delta + \lambda^2_+ \sigma^\circ(x-x_0)^2/4\delta \right)
	+ u(x_0, s)
\]
is a supersolution with
\[
	u(\cdot, s) \leq v(\cdot, s).
\]
By the standard comparison theorem \cite[Chapter 3]{GigaBook}, we conclude that
\[
	u(x,t) \leq v(x,t) \quad\text{for}\quad t \geq s, \quad
	x \in \mathbb{R}^n.
\]
In particular,
\[
	u(x_0,t) \leq L_s \left( C(t-s)/\delta+\delta \right)
	+ u(x_0,s)
\]
We take $\delta$ such that $\delta=\left( C(t-s) \right)^{1/2}$ to get
\[
	u(x_0,t) - u(x_0,s) \leq 2L_s \left( C(t-s)\right)^{1/2}
	\quad\text{for all}\quad t \geq s>0.
\]
A symmetric argument yields the estimate from below.
 Since $x_0,s$ are arbitrary, this completes the proof.
\end{proof}

We are ready to prove the existence of solutions of \eqref{pde} for $F$ of the form \eqref{level-set-F}.

\begin{theorem}
\label{th:existence}
Let $F$ be of the form \eqref{level-set-F}, where $g$ and $f$ are as in Theorem~\ref{L} and assume that $\sigma$ is a crystalline anisotropy. Then the equation \eqref{pde} has a unique global viscosity solution on $\Rn \times (0, \infty)$ for any Lipschitz initial data $u_0$ constant outside a large ball $B$. Moreover, $u$ is Lipschitz continuous in space \eqref{u-Lip-space} and 1/2-H\"{o}lder continuous in time \eqref{u-Holder-time}.
\end{theorem}

\begin{proof}
We can approximate $\sigma$ by a sequence of positively one-homogeneous functions $\sigma_m$ as in Theorem~\ref{th:stability-linear-growth}, so that they all satisfy the assumptions of Theorem~\ref{H} with the same $\lambda_\pm$.
By the classical viscosity solution theory, \eqref{regularized-problem} has a unique viscosity solution $u_m$ with the initial data $u_0$ for any $m$. For any $T > 0$, the sequence $\set{u_m}$ is uniformly Lipschitz in space by Theorem~\ref{L} and uniformly H\"{o}lder in time by Theorem~\ref{H}. Hence any subsequence has a further subsequence that converges uniformly on $\Rn \times [0, T]$ to some continuous function $u$ with $u(\cdot, 0) = u_0$. This limit must be the unique viscosity solution of \eqref{pde} by Theorem~\ref{th:stability-linear-growth}. In particular, the whole sequence converges to $u$ locally uniformly on $\Rn \times (0, \infty)$.
\end{proof}

\begin{proof}[Sketch of the proof of Theorem~\ref{th:existence-flow}]
Once we have existence, comparison and stability of solutions of \eqref{level-set-F}, we are able to prove invariance under the change of depending variables as in \cite{GP_ADE,GP_CPAM}. This yields the uniqueness of the level set of a solution. Our existence result (Theorem~\ref{th:existence}) now yields the unique existence of the level set flow.
\end{proof}

\section{Proof of nonexistence with \texorpdfstring{$x$-dependent $F$}{x-dependent F}}
\label{sec:proof-of-nonexistence}

In this section we show that we cannot allow the operator $F$ to depend on $x$ directly since in that case a solution might not exist in general. Let us consider the equation
\begin{align}
\label{sign-problem}
u_t + F\big(x, (\sign u_x)_x\big) = 0, \qquad x \in \R, t > 0,
\end{align}
with $F(x, \xi) = -\xi + f(x)$ and initial condition
\begin{align*}
u(\cdot, 0) = 0.
\end{align*}
This corresponds to anisotropy $\sigma(p) = |p|$ for $p \in \R$.

\begin{theorem}
\label{th:nonexistence}
Let $f \in Lip(\R)$ with compact support, $f \geq 0$ but $f \not\equiv 0$ (see Figure~\ref{fig:nonexistence}(a)).  Set $L = 1 / \max f$. Assume that
\begin{align*}
\supp f \subset (-L, L).
\end{align*}
Then there is no continuous solution with compact support.
\end{theorem}

\begin{proof}
We show it by contradiction with the comparison principle, Theorem~\ref{th:comparison-principle}, which applies to equation~\eqref{sign-problem}.

Let $u$ be a solution of \eqref{sign-problem} with initial data $0$ and compact support. Let us fix $\omega > L$. For given $a > 0$ set
\begin{align*}
v(x) = v_a(x) := \min\Big(\max\big(a(|x| - \omega), a(L - \omega)\big), 0\Big).
\end{align*}
Note that $v$ has a facet $[-L, L]$ that contains the support of $f$.
We claim that $w(x,t) = v_a(x)$ is a viscosity subsolution in the sense of Definition~\ref{def:visc-solution}. Indeed, for any faceted test function that touches $w$ from above on the facet $[-L, L]$, we have $\Lambda \geq \frac 1L$ and therefore
\begin{align*}
F(x, \Lambda) = -\Lambda + f \leq -\frac 1L + \max f = -\max f + \max f = 0.
\end{align*}
For a test function touching anywhere else we always get $F(x, \Lambda) \leq 0$.

\medskip
Therefore by the comparison principle we must have $v_a \leq u(\cdot, t)$ for all $t$ and $a > 0$ and so $u \geq 0$. On the other hand, $f \geq 0$ and therefore $u = 0$ is a supersolution so we conclude that $u \equiv 0$.

But by testing $u$ from above by a faceted test function with a facet longer than $L$, we can show that there are points where $F(x, \Lambda) = -\Lambda + f(x) > -\frac1L + f(x) = 0$ and so we see that $u \equiv 0$ is not a viscosity subsolution in the sense of Definition~\ref{def:visc-solution}.
\end{proof}
In fact, this shows that a supremum of subsolutions might not be a subsolution.

\medskip

One may be interested in what the solution is for $u_t + F((\sign u_x)_x - f(x)) = 0$ with $F(\xi) = - \xi$. Such a type of problems is studied in the framework of the maximal monotone operators \cite{GG2}. Although a Lipschitz bound is studied only for level set equations, it is not difficult to show uniform Lipschitz bounds and also $1/2$-H\"older continuity in time for the $\sigma_m$-approximation as we did in Section~\ref{sec:lipschitz-bound}. By our convergence results we conclude that the viscosity solution agrees with that in the theory of maximum monotone operators.

\begin{figure}
\centering
\begin{tikzpicture}
\def\myscale{2}
\begin{scope}[scale=\myscale,very thin]
\draw[->] (-1.4,0) -- (1.4,0) node[below] {$x$};
\draw[->] (0,-0.2) -- (0,1.2) node[left] {$f$};
\draw (1,2pt/\myscale) -- node[below] {$R$} (1,-2pt/\myscale);
\draw (-1,2pt/\myscale) -- node[below] {$-R$} (-1,-2pt/\myscale);
\draw[thin,domain=-1:1,smooth] plot (\x, {-((\x-1)*(\x+1)*(exp(-6.5*6.5*\x*\x)+0.9))/1.9});
\draw (0,-2/\myscale) node {(a)};
\end{scope}
\def\myscale{0.5}
\begin{scope}[xshift=2.4in,xscale=1.4,yscale=\myscale,very thin]
\draw[->] (-1.8,0) -- (2,0) node[below] {$x$};
\draw[->] (0,-3.2) -- (0,3.2);
\draw (1.57,2pt/\myscale) -- node[below] {$R$} (1.57,-2pt/\myscale);
\draw (-1.57,2pt/\myscale) -- node[below] {$-R$} (-1.57,-2pt/\myscale);
\draw (1.2,2pt/\myscale) -- node[below] {$\ell$} (1.2,-2pt/\myscale);
\draw (-1.2,2pt/\myscale) -- node[below] {$-\ell$} (-1.2,-2pt/\myscale);
\draw[thick,domain=-1.57:1.57,smooth] plot (-\x, {-2*sin(deg(-\x))+1}) node[left] {$Z_+$};
\draw[thick,domain=-1.57:1.57,smooth] plot (-\x, {-2*sin(deg(-\x))-1}) node[left] {$Z_-$};
\draw[thin,dashed] (-1.2, {-2*sin(deg(-1.2))-1}) -- node[below right=5pt] {$y$} (1.2, {-2*sin(deg(1.2)) + 1});
\draw (0,-2/\myscale) node {(b)};
\end{scope}
\end{tikzpicture}
\caption{(a) An example of the nonuniform forcing $f$ in Theorem~\ref{th:nonexistence}. (b) The construction of $\chi_\ell$.}
\label{fig:nonexistence}
\end{figure}
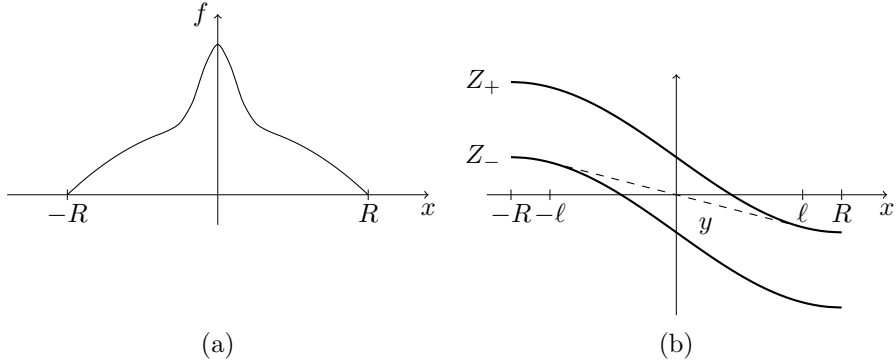

We here give an explicit example of the solution. We consider $f \in Lip(\R)$ with compact support, $f \geq 0$ but $f \not\equiv 0$. To simplify the explanation we assume that $f$ is even, that is, $f(x) = f(-x)$, $x\in \R$, and that $f$ is non-increasing for $f >0$. We take $R >0$ such that $\supp f = [-R, R]$. Let $Z_\pm$ be
\begin{align*}
Z_\pm(x) = -\int_0^x f(z)\,dz \pm 1.
\end{align*}
Since $Z_\pm$ is convex in $x$ for $x > 0$ and $Z_\pm-\pm 1$ is odd, there is a unique $\ell \in (0, R)$ such that the straight segment $y = y(x)$ connecting $(-\ell, Z_-(-\ell))$, $(\ell, Z_+(\ell))$ has the slope $Z'_+(\ell) = -f(\ell)$; see Figure~\ref{fig:nonexistence}(b). Moreover,
\begin{align*}
Z_-(x) \leq y(x) \leq Z_+(x)
\end{align*}
for $x \in [-\ell, \ell]$. By definition of $\Lambda_f$, we observe as in \cite{GG2}
\begin{align*}
\Lambda_f[\chi_\ell] = y'(x) = \frac{Z_+(\ell) - Z_-(-\ell)}{2\ell},
\end{align*}
if $\chi_\ell$ corresponds to the facet
\begin{align*}
\chi_\ell(x) :=
\begin{cases}
1, & x\notin [-\ell, \ell],\\
0, & x \in [-\ell, \ell].
\end{cases}
\end{align*}
One is able to characterize $\ell$ in a slightly different way since $y'(\ell) = -f(\ell)$. Since $\ell$ is the unique number such that
\begin{align*}
-Z_+'(\ell) = \frac{Z_-(-\ell) - Z_+(\ell)}{2\ell},
\end{align*}
it is the unique number satisfying
\begin{align}
\label{fell-cond}
f(\ell) = -\frac 1\ell +\frac 1{2\ell} \int_{-\ell}^\ell f(x)\,dx.
\end{align}
We consider the initial value problem
\begin{align}
\label{sign-p}
u_t + F\big((\sign u_x)_x - f(x)\big) = 0
\end{align}
with $F(\xi) = -\xi$. By using above $\ell$ one gets an explicit solution.

\begin{theorem}
Assume that $f$ is as above. Let $\ell \in (0, R)$ be a unique number satisfying \eqref{fell-cond}. Then the function
\begin{align*}
u(x,t) = \max(-f(x)t, -f(\ell) t)
\end{align*}
is the unique solution of \eqref{sign-p} with initial data $u_0 \equiv 0$.
\end{theorem}

Since the facet is $\chi_\ell$ for all $t > 0$, it is easy to see that this is a solution. For more examples of solutions see \cite{GG2}.

\appendix
\section{Regularity of the solution of the resolvent problem}

\label{sec:resolvent-problem}

In the perturbed test function method in the proof of Theorem~\ref{th:stability}, we need to solve the resolvent problem for $\E_f$ defined in \eqref{Ef}. Here we show that the solution of this resolvent problem is Lipschitz continuous. Suppose that $\psi$ is a Lipschitz function with Lipschitz constant $L$ and $f$ is a Lipschitz function with Lipschitz constant $L_f$.
Consider the solution $\zeta$ of
\begin{align*}
\zeta + a \partial \E_f(\zeta) \ni \psi.
\end{align*}
By shifting $\zeta_y := \zeta(\cdot - y)$, etc., we have
\begin{align*}
\zeta_y + a \partial \E_{f_y}(\zeta_y) \ni \psi_y,
\end{align*}
or, equivalently,
\begin{align*}
\zeta_y + a \partial \E_0(\zeta_y) \ni \psi_y - a f_y.
\end{align*}
By the Lipschitz continuity, we have
\begin{align*}
\psi_y - a f_y - (L + aL_f)|y| \leq \psi - a f \leq \psi_y - a f_y + (L + aL_f)|y|
\end{align*}
and hence the comparison principle for the elliptic resolvent problem yields
\begin{align*}
\zeta_y - (L+aL_f)|y| \leq \zeta \leq \zeta_y + (L+aL_f)|y|.
\end{align*}
We conclude that $\zeta$ is Lipschitz continuous with the Lipschitz constant $L +aL_f$.

\subsection*{Acknowledgments}
The first author was partly supported by the Japan Society for the Promotion of Sciences (JSPS) through Grant-in-Aid for Scientific Research Kiban A (No. 19H00639), Kiban A (No. 17H01091) and Challenging Pioneering Research (Kaitaku)(No. 18H05323).
The second author was supported by JSPS KAKENHI Wakate Grant (No. 18K13440).

\end{document}